\documentclass{article}
\usepackage[textwidth=18cm]{geometry}
\usepackage{amsmath,amssymb,amsthm}
\usepackage{graphicx}
\usepackage[pdfborder= 0 0 0,citecolor=black,linkcolor=black,colorlinks=true,bookmarksopen=true]{hyperref}
\usepackage{caption}
\usepackage{subcaption}
\usepackage{algorithm}
\usepackage{algorithmicx}
\usepackage{algpseudocode}
\usepackage{authblk}
\usepackage[square]{natbib}
\usepackage[pagewise]{lineno}
\usepackage[colorinlistoftodos,prependcaption]{todonotes}
\usepackage{draftwatermark}
\SetWatermarkScale{3}
\SetWatermarkLightness{0.9}
\SetWatermarkText{For review SIAM UQ}

\author[1,*]{P.A. Browne}
\affil[1]{Department of Meteorology, University of Reading, UK}
\affil[*]{Correspondence to p.browne@reading.ac.uk}
\title{Model error moment estimation via data assimilation}
\makeatletter \AtBeginDocument{ \hypersetup{pdftitle= {\@title},pdfauthor= {\@author}}} \makeatother

\date{\today}

\newcommand{\wtm}{\mbox{$\widetilde{\mu}$}}
\theoremstyle{plain}

\theoremstyle{definition}

\theoremstyle{theorem}
\newtheorem{theorem}{Theorem}
\newtheorem{corollary}[theorem]{Corollary}
\newtheorem{remark}[theorem]{Remark}
\newtheorem{lemma}[theorem]{Lemma}
\makeatletter
\begin{document}

\maketitle

\begin{abstract}
Using a dynamical model to make predictions about a system has many sources of error. These can include errors in how the model was initialised but also errors in the dynamics of the model itself.
For many applications in data assimilation, probabilistic forecasting, or model improvement, these model errors need to be known over the timestep of the model, not over a time-averaged period. 
Using a forecast from a state that combines observational information as well as prior information we can gain an approximation to the statistics of the model errors on the timescale of the model that is required.
Here we give bounds on the errors in the estimation of the mean and covariance of the errors in the model equations in terms of the errors made in the state estimation.
This is the first time that such a result has been derived. The result shows to what extent the state estimation must constrain the analysis in order to obtain a specified error on the mean or covariance of the model errors.
This is particularly useful for experimental design as it indicates the necessary information content required in observations of the dynamical system.

 \end{abstract}

\section{Introduction}
Suppose we have a numerical model, $f$, for a dynamical process:
\begin{equation}
x^{k+1} = f(x^k)
\end{equation}
where $x \in \mathbb{R}^N$ is the state at time indexed by $k$. Then it is almost certain that such a numerical model $f$ will contain errors. 
A large amount of uncertainty quantification has revolved around providing bounds of the accuracy of the numerical scheme used to approximate the underlying mathematical or statistical problem \citep{Owhadi2013,Teckentrup2013}. This ignores one large issue: the underlying mathematical problem does not represent all of the features of the real world system in which one is interested. For example, when modelling wind around a wind farm, drag induced by individual trees/plants will likely not be included. Instead, some approximate homogenised quantity will be used. As another example, when modelling carbon fibre composites, specific manufacturing defects will not be included in the model until measurements of the materials response under various loading conditions are incorporated.

The ubiquitous quote on this issue is from \citet{Box1987}: ``All models are wrong but some are useful''. The problem we  address is how to estimate the distribution of the errors that are made by the mathematical model in simulating a physical system.

To get a measure of how far a mathematical model is deviating from the real world, it is necessary to use observations from the system as an independent source of information. Typically we shall never be able to fully observe a system (in time and/or space) and so one has to be careful how to compare a small set of observations with the larger model of the system. The theoretical framework to do this rigorously is well established: using data assimilation to numerically implement Bayes' theorem \citep[][for example]{Jazwinski1970}.

A physical system that occurs in nature has only one realisation. We refer to that realisation as the \emph{truth} and denote it $x_t$. If we use such a truth with the numerical model of the system we have the following evolution equation,
\begin{equation}\label{eqn:truth}
x^{k+1}_t = f(x^k_t) + \beta^k.
\end{equation}
Here $\beta^k$ is the model error term which is a realisation of a stochastic process which occurs at time $k$. We wish to estimate, through the use of data assimilation, the properties of the distribution that $\beta^k$ follows. The first two moments of this unknown distribution, its mean and covariance, are denoted $\mu$ and $Q$ respectively. Knowledge of $\beta^k$ can be fed back into the dynamical model $f$ as a way to systematically improve the model. For example, directly subtracting $\mu$ from the model would be an example of a bias correction method \citep[e.g.][]{Dee2009}, whereas the sources of these errors can be attributed to components of the model equations \cite{Lang2016} and thus can be corrected to improve the mathematical model.

We use any or all prior information we have, along with observations of the system to obtain a \emph{``best estimate''} of the truth. This is the data assimilation stage and results in what is referred to as the \emph{``analysis''}, $x_a$. There are many different techniques for data assimilation, each which make various assumptions or simplifications to Bayes' theorem. See \citet{Evensen2007} or \citet{law2015} for an overview of the field.

Once we have $x_a$ we define our approximation to the model error at timestep $k$, $\widetilde{\beta}^k$, with the equation
\begin{equation}\label{eqn:jobber}
x^{k+1}_a = f(x^k_a) + \widetilde{\beta}^k.
\end{equation}
This makes no assumptions on how $x_a$ is calculated. However, we can immediately see that one traditional method of data assimilation is inappropriate to estimate the $\beta^k$ term.

\begin{remark}
To estimate $\beta^k$, $x_a$ should not be obtained from strong constraint 4DVar (SC-4DVar).
\end{remark}
In SC-4DVar the analysis at time $0$, $x_a^0$ is given as 
\begin{equation}
x_a^0 = \arg \min_{x^0} ||x^0 - x_b||^2_{B^{-1}} + \sum_{i=1}^m||y_{k_i}-H_{k_i}(f^{k_i}(x^0))||^2_{R_{k_i}^{-1}}.
\end{equation}
Here, $x_b$ is a background estimate for the state at time $0$, assumed to be Gaussian with background error covariance matrix $B$. Observations of the system $y$ are taken at $m$ separate times $k_i$. The observation is related to the system through the observation equation
\[y_{k_i} = H_{k_i}(x_t^{k_i}) + \eta^{k_i},\]
where the observation errors $\eta^{k_i}$ are assumed to be Gaussian with an observation error covariance matrix $R_{k_i}$. Then at subsequent times in the assimilation window, 
\begin{equation}
x^{k+1}_a = f(x^k_a) = f^{k+1}(x_a^0)
\end{equation}
and thus $\widetilde{\beta}^k = 0 \ \forall  k \in \{ 0,\ldots,\tau-1\}$.

\subsection{Usage of the model error covariance matrix $Q$}

Having an estimate for $Q$ is important as it plays a vital role in many modern data assimilation methods. In the variational method of data assimilation known as weak constraint 4DVar (WC-4DVar), in addition to the background and observation terms of SC-4DVar, a third term is introduced to penalise the addition of model error \citep{Jazwinski1970}, which is implicitly assumed to follow a Gaussian distribution.
\begin{equation}
\min_{x^0,\beta^0,\ldots,\beta^{k-1}} ||x^0 - x_b||^2_{B^{-1}} + \sum_{i=1}^m||y_{k_i}-H_{k_i}(x^0,\beta^0,\ldots,\beta^{k_i-1})||^2_{R_{k_i}^{-1}} + \sum_{j=0}^{k-1}||\beta^j||^2_{Q^{-1}}
\end{equation}
In this formulation, $x^0$ is the initial state which the method seeks to find, along with model error terms $\beta^0,\ldots,\beta^{k-1}$.
The first term penalises the difference in the initial state $x_0$ and an initial guess known as the background term $x_b$, weighted by the inverse of the background error covariance matrix $B$. The second term penalises discrepancies between observations of the system $y_{k_i}$ and the model equivalent of the observations at the corresponding time, $H_{k_i}(x^0,\beta^0,\ldots,\beta^{k_i-1})$, all weighted by the appropriate observation error covariance matrix $R_{k_i}$. This notation hides the fact that $H_{k_i}(x^0,\beta^0,\ldots,\beta^{k_i-1})=H_{k_i}(x^{k_i}) = H_{k_i}(f(x^{k_i-1})+\beta^{k_i-1})=H_{k_i}(f(f(x^{k_i-2})+\beta^{k_i-2})+\beta^{k_i-1}) = \ldots $ etc. The final term penalises the introduction of model error at each model timestep $j$, weighted by the model error covariance matrix $Q$.

In particle filters, such as the equivalent weights particle filter \citep{VanLeeuwen2010}, $Q$ is used as part of the proposal density to bring the ensemble closer to the observations. It also appears in the computation of weights associated with each particle. \citet{Browne2015a} showed that the lack of information regarding $Q$ is the limiting factor to applying particle filters in large-scale geophysical systems.

\subsection{Previous work on model error estimation}

Much of the work to estimate $Q$ has been undertaken at the European Centre for Medium-range Weather Forecasting (ECMWF) in the context of implementing a form of WC-4DVar for operational numerical weather forecasting. The lack of knowledge of $Q$ is one of the main reasons that such weak-constraint methods are not operational. \citet{Tremolet2006} notes ``Weak-constraint 4D-Var has never been implemented fully with a realistic forecast model because of the computational cost and because of the lack of information to define the model-error covariance matrix required to solve the problem''.

It has been shown \citep{Tremolet2007} 
that for the atmospheric case, the background error covariance matrix ($B$) is not successful as an approximation to the model-error covariance matrix ($Q$). As an \textit{ad hoc} approximation to $Q$, \citet{Tremolet2007} suggests the following.
At time $k$, get an ensemble of analysis states $\{x_i^k\}$. As each ensemble member should be indistinguishable from the truth, each forecast $f(x_i^k)$ could be considered a ``possible evolution of the atmosphere from the true state''. From this, Tr\'emolet concluded that the differences in these forecast increments could be interpreted as ``an ensemble of possible realisations of model error''. The real world, however, has only one realisation. Hence one needs to compare these forecasts not with themselves, but with the observations instead. 
In this way, one discovers not just where the model diverges based on its initial conditions, but where it has failed to capture the realisation of the underlying stochastic dynamics.

A smoother is a sequential method of data assimilation that gives an estimate of $x_a$ for all $k$, not simply at timesteps where observations are present.
Recently, the concept of using a forecast from a lag-1 smoother estimate to approximate the model error covariances has been investigated \citep{Todling2014}. This work only considers estimating the $Q$  the problem in observation space, and do not look at the full information available to them in state space. This is physically motivated as in the case where only the observed variables of the dynamical system are well constrained, only the projection of $Q$ into the observation space will be accurate.

Such derivations from a smoother trajectory have been investigated by \citet{Lang2016} in the context of parameter and parameterization estimation. When the model error can be attributed to incorrect coefficients in the formulation of the model equations then it has been shown that the first moment of the derived model error can be used to determine the necessary corrections to the parameters used in the numerical model. This can be extended to estimation of parameterizations when the first moment of the model error is regressed onto different functions of the model state vector. However higher order moments have not been investigated.

\section{Error bounds on the estimation of the first two moments of model error}

We now set out on the main path of this paper: to find a bound on the error in the approximation of $\mu$ and $Q$ in terms of the accuracy of $x_a$ in approximating $x_t$. The main assumption we shall make is that the dynamical model $f$ is (Lipschitz) continuous. If the model were not continuous at $x_t^k$ then regardless of how close we approximate this argument, the resulting model trajectory could be wildly different. This would lead to a poor approximation of the model error.

\begin{theorem}\label{thm:beta}Suppose $f$ is Lipschitz with constant $L$ and $|x_t^k - x_a^k| < \frac{\varepsilon}{L+1}\ \forall k$.
Then $|\widetilde{\beta}^k - \beta^k| < \varepsilon$.

\end{theorem}
\begin{proof}
$f$ Lipschitz $\implies \exists L>0$ s.t. $|f(a)-f(b)|\le L|a-b| \ \forall a,b$.

\begin{align*}
\beta^k = x_t^{k+1} - f(x_t^k),\\
\widetilde{\beta}^k = x_a^{k+1} - f(x_a^k).
\end{align*}
Hence
\begin{align*}
|\beta^k-\widetilde{\beta}^k| &= |x_t^{k+1}-x_a^{k+1} - f(x_t^k) +f(x_a^k)|\\
&\le |x_t^{k+1}-x_a^{k+1}| + |f(x_t^k) -f(x_a^k)|.
\end{align*}
As $|x_t^n - x_a^n| < \frac{\varepsilon}{L+1}$ then
\begin{align*}
|\beta^k-\widetilde{\beta}^k| &\le \frac{\varepsilon}{L+1} + \frac{\varepsilon L}{L+1} = \varepsilon.
\end{align*}
\end{proof}

The sample mean of the model error, $\mu$, is calculated simply as
\begin{equation}
\mu = \frac{1}{\tau}\sum_{k=1}^\tau \beta^k.
\end{equation}
The approximated sample mean of the model error, $\tilde{\beta}^k$, in similarly calculated as
\begin{equation}
\widetilde{\mu} = \frac{1}{\tau}\sum_{k=1}^\tau \widetilde{\beta}^k.
\end{equation}
Using these definitions and Theorem \ref{thm:beta}, we can show the following bound on the estimation of the mean of the model error distribution.

\begin{corollary}\label{cor:mu}
Suppose $f$ is Lipschitz and $|x_t^k - x_a^k| < \frac{\varepsilon}{L+1}\ \forall k$. Then $|\widetilde{\mu} - \mu| < \varepsilon$.
\end{corollary}
\begin{proof}
\begin{align}
|\mu - \widetilde{\mu}| &= \left|\frac{1}{\tau}\sum_{k=1}^\tau \beta^k - \frac{1}{\tau}\sum_{k=1}^\tau \widetilde{\beta}^k\right| \\
			&= \frac{1}{\tau}\left|\sum_{k=1}^\tau \beta^k - \sum_{k=1}^\tau \widetilde{\beta}^k\right| \\
			&= \frac{1}{\tau}\left|\sum_{k=1}^\tau (\beta^k - \widetilde{\beta}^k)\right| \\
			&\le \frac{1}{\tau}\sum_{k=1}^\tau \left|\beta^k - \widetilde{\beta}^k\right| \\
			&\le \frac{1}{\tau}\sum_{k=1}^\tau \varepsilon \\
			&= \varepsilon
\end{align}
\end{proof}

We now consider the second moment of the model error distribution. 
The sample covariance of model error is defined and approximated as, respectively,
\begin{equation}
Q = \frac{1}{\tau-1}\sum_{k=1}^\tau (\beta^k - \mu)(\beta^k - \mu)^T \qquad \text{ and } \qquad \widetilde{Q} = \frac{1}{\tau-1}\sum_{k=1}^\tau (\widetilde{\beta}^k - \widetilde{\mu})(\widetilde{\beta}^k - \widetilde{\mu})^T.
\end{equation}

\begin{lemma}\label{lemma:prod}
If $| f - L | < \min\{1,\frac{\varepsilon}{2|M|}\}$ and $|g - M | < \min\{1,\frac{\varepsilon}{2|L|+2}\}$ then $|fg-LM| < \varepsilon$.
\end{lemma}
\begin{proof}
By the triangle inequality
\[ |f| - |L| \le |f-L| < 1 \implies |f| < 1 + |L|. \]
Hence
\begin{align}
|fg-LM| &= |f(g-M) + M(f-L)|\\
	&\le |f||g-M| + |M||f-L|\\
	& <  (1+|L|)\frac{\varepsilon}{2|L|+2} + |M|\frac{\varepsilon}{2|M|}\\
	& <  \frac{\varepsilon}{2} + \frac{\varepsilon}{2} = \varepsilon.
\end{align} 
\end{proof}

Introducing a subscript notation where $\cdot_i$ refers to the $i^{\text{th}}$ component of a vector and $\cdot_{ij}$ refers to the $i^{\text{th}},j^{\text{th}}$ component of a matrix, we show the following preliminary result.

\begin{lemma}\label{lemma:q}
Suppose $f$ is Lipschitz with constant $L$ and for all $k \in \{1,\ldots,\tau\}$
\begin{equation}
|\beta^k_i - \widetilde{\beta}^k_i | < \min \left\{ 1, \frac{\varepsilon}{8|\beta^k_j|}\left(\frac{\tau-1}{\tau}\right), \frac{\varepsilon}{8|\mu_j|}\left(\frac{\tau-1}{\tau}\right)\right\},
\end{equation}
\begin{equation}
|\beta^k_j - \widetilde{\beta}^k_j | < \min \left\{ 1, \frac{\varepsilon}{8|\beta^k_i|+8}\left(\frac{\tau-1}{\tau}\right), \frac{\varepsilon}{8|\mu_i|+8}\left(\frac{\tau-1}{\tau}\right)\right\},
\end{equation}
\begin{equation}
|\mu_i - \widetilde{\mu}_i| < \min\left\{1, \frac{\varepsilon}{8|\beta^k_j|}\left(\frac{\tau-1}{\tau}\right), \frac{\varepsilon}{8|\mu_j|}\left(\frac{\tau-1}{\tau}\right)\right\},
\end{equation}
\begin{equation}
|\mu_j - \widetilde{\mu}_j| < \min\left\{1, \frac{\varepsilon}{8|\beta^k_i|+8}\left(\frac{\tau-1}{\tau}\right), \frac{\varepsilon}{8|\mu_i|+8}\left(\frac{\tau-1}{\tau}\right)\right\}.
\end{equation}

Then $|Q_{ij} - \widetilde{Q}_{ij}| < \varepsilon$.
\end{lemma}

\begin{proof}
\begin{align}
Q_{ij} - \widetilde{Q}_{ij} &= \frac{1}{\tau-1} \sum_{k=1}^\tau
\left[
(\beta^k_i - \mu_i)(\beta^k_j - \mu_j) - (\widetilde{\beta}^k_i - \wtm_i)(\widetilde{\beta}^k_j - \wtm_j)
\right]\\
&=\frac{1}{\tau-1} \sum_{k=1}^\tau
\left[
(\beta^k_i\beta^k_j - \widetilde{\beta}^k_i\widetilde{\beta}^k_j) - 
(\beta^k_i\mu_j - \widetilde{\beta}^k_i\wtm_j) -
(\mu_i\beta^k_j - \wtm_i\widetilde{\beta}^k_j) +
(\mu_i\mu_j - \wtm_i\wtm_j)
\right]
\end{align}
Hence
\begin{align}
\left|Q_{ij} - \widetilde{Q}_{ij}\right| &\le \frac{1}{\tau-1} \sum_{k=1}^\tau
\left[
\left|\beta^k_i\beta^k_j - \widetilde{\beta}^k_i\widetilde{\beta}^k_j\right| - 
\left|\beta^k_i\mu_j - \widetilde{\beta}^k_i\wtm_j\right| -
\left|\mu_i\beta^k_j - \wtm_i\widetilde{\beta}^k_j\right| +
\left|\mu_i\mu_j - \wtm_i\wtm_j\right|
\right]
\end{align}

By Lemma \ref{lemma:prod}, if
\begin{equation}
| \widetilde{\beta}^k_i - \beta^k_i | < \min\left\{1,\frac{\varepsilon}{8|\beta^k_j|}\left(\frac{\tau-1}{\tau}\right)\right\}
\qquad
\text{and} 
\qquad
|\widetilde{\beta}^k_j - \beta^k_j | < \min\left\{1,\frac{\varepsilon}{8|\beta^k_i|+8}\left(\frac{\tau-1}{\tau}\right)\right\}
\end{equation}
then
\begin{equation}
\left|\beta^k_i\beta^k_j - \widetilde{\beta}^k_i\widetilde{\beta}^k_j\right| < \frac{\varepsilon}{4}\left(\frac{\tau-1}{\tau}\right).
\end{equation}
If
\begin{equation}
| \widetilde{\mu}_i - \mu_i | < \min\left\{1,\frac{\varepsilon}{8|\mu_j|}\left(\frac{\tau-1}{\tau}\right)\right\}
\qquad
\text{and} 
\qquad
|\widetilde{\mu}_j - \mu_j | < \min\left\{1,\frac{\varepsilon}{8|\mu_i|+8}\left(\frac{\tau-1}{\tau}\right)\right\}
\end{equation}
then
\begin{equation}
\left|\mu_i\mu_j - \widetilde{\mu}_i\widetilde{\mu}_j\right| < \frac{\varepsilon}{4}\left(\frac{\tau-1}{\tau}\right).
\end{equation}
If
\begin{equation}
| \widetilde{\beta}^k_i - \beta^k_i | < \min\left\{1,\frac{\varepsilon}{8|\mu_j|}\left(\frac{\tau-1}{\tau}\right)\right\}
\qquad
\text{and} 
\qquad
|\widetilde{\mu}_j - \mu_j | < \min\left\{1,\frac{\varepsilon}{8|\beta^k_i|+8}\left(\frac{\tau-1}{\tau}\right)\right\}
\end{equation}
then
\begin{equation}
\left|\beta^k_i\mu_j - \widetilde{\beta}^k_i\wtm_j\right| < \frac{\varepsilon}{4}\left(\frac{\tau-1}{\tau}\right).
\end{equation}
Similarly, if
\begin{equation}
| \widetilde{\mu}_i - \mu_i | < \min\left\{1,\frac{\varepsilon}{8|\beta^k_j|}\left(\frac{\tau-1}{\tau}\right)\right\}
\qquad
\text{and} 
\qquad
|\widetilde{\beta}^k_j - \beta^k_j | < \min\left\{1,\frac{\varepsilon}{8|\mu_i|+8}\left(\frac{\tau-1}{\tau}\right)\right\}
\end{equation}
then
\begin{equation}
\left|\mu_i\beta^k_j - \wtm_i\widetilde{\beta}^k_j\right| < \frac{\varepsilon}{4}\left(\frac{\tau-1}{\tau}\right).
\end{equation}

Thus we can combine these so that
\begin{align}
\left|Q_{ij} - \widetilde{Q}_{ij}\right| &< \frac{1}{\tau-1} \sum_{k=1}^\tau
\left[
\frac{\varepsilon}{4}\left(\frac{\tau-1}{\tau}\right) +
\frac{\varepsilon}{4}\left(\frac{\tau-1}{\tau}\right) +
\frac{\varepsilon}{4}\left(\frac{\tau-1}{\tau}\right) +
\frac{\varepsilon}{4}\left(\frac{\tau-1}{\tau}\right)
\right] < \varepsilon.
\end{align}
\end{proof}

It is now possible to obtain a bound on the error in the covariance of the model error distribution in terms of the analysis error. For clarity, the notation $\cdot_{,i}$ also refers to the $i^{\text{th}}$ component of a vector that already has a subscript.
\begin{theorem}\label{thm:q}
Suppose $f$ is Lipschitz with constant $L$. Suppose further that, for all $k$, the analysis satisfies
\begin{equation}
|x^k_{t,i} - x^k_{a,i}| < \frac{1}{L+1}\min \left\{ 1, \frac{\varepsilon}{8|\beta^k_j|}\left(\frac{\tau-1}{\tau}\right), \frac{\varepsilon}{8|\mu_j|}\left(\frac{\tau-1}{\tau}\right)\right\}
\end{equation}
and 
\begin{equation}
|x^k_{t,j} - x^k_{a,j}| < \frac{1}{L+1}\min \left\{ 1, \frac{\varepsilon}{8|\beta^k_i|+8}\left(\frac{\tau-1}{\tau}\right), \frac{\varepsilon}{8|\mu_i|+8}\left(\frac{\tau-1}{\tau}\right)\right\}
\end{equation}
Then the corresponding estimated sample error covariance satisfies $|Q_{ij} - \widetilde{Q}_{ij}| < \varepsilon$.
\end{theorem}

\begin{proof}
By Theorem \ref{thm:beta}, if
\begin{equation}
|x^k_{t,i} - x^k_{a,i}| < \frac{1}{L+1}\min \left\{ 1, \frac{\varepsilon}{8|\beta^k_j|}\left(\frac{\tau-1}{\tau}\right), \frac{\varepsilon}{8|\mu_j|}\left(\frac{\tau-1}{\tau}\right)\right\}
\end{equation}
then
\begin{equation}
|\widetilde{\beta}^k_i - \beta^k_i| < \min \left\{ 1, \frac{\varepsilon}{8|\beta^k_j|}\left(\frac{\tau-1}{\tau}\right), \frac{\varepsilon}{8|\mu_j|}\left(\frac{\tau-1}{\tau}\right)\right\}.
\end{equation}
Similarly if 
\begin{equation}
|x^k_{t,j} - x^k_{a,j}| < \frac{1}{L+1}\min \left\{ 1, \frac{\varepsilon}{8|\beta^k_i|}\left(\frac{\tau-1}{\tau}\right), \frac{\varepsilon}{8|\mu_i|}\left(\frac{\tau-1}{\tau}\right)\right\}
\end{equation}
then
\begin{equation}
|\widetilde{\beta}^k_j - \beta^k_j| < \min \left\{ 1, \frac{\varepsilon}{8|\beta^k_i|}\left(\frac{\tau-1}{\tau}\right), \frac{\varepsilon}{8|\mu_i|}\left(\frac{\tau-1}{\tau}\right)\right\}.
\end{equation}

By Corollary \ref{cor:mu} if 
\begin{equation}
|x_{t,i}^k - x_{a,i}^k| < \frac{1}{L+1}\min\left\{1, \frac{\varepsilon}{8|\beta^k_j|}\left(\frac{\tau-1}{\tau}\right), \frac{\varepsilon}{8|\mu_j|}\left(\frac{\tau-1}{\tau}\right)\right\}\ \forall k.
\end{equation}
then
\begin{equation}
|\widetilde{\mu}_i - \mu_i| < \min\left\{1, \frac{\varepsilon}{8|\beta^k_j|}\left(\frac{\tau-1}{\tau}\right), \frac{\varepsilon}{8|\mu_j|}\left(\frac{\tau-1}{\tau}\right)\right\}
\end{equation}
and if
\begin{equation}
|x_{t,j}^k - x_{a,j}^k| < \frac{1}{L+1}\min\left\{1, \frac{\varepsilon}{8|\beta^k_i|+8}\left(\frac{\tau-1}{\tau}\right), \frac{\varepsilon}{8|\mu_i|+8}\left(\frac{\tau-1}{\tau}\right)\right\}
\end{equation}
then
\begin{equation}
|\widetilde{\mu}_j - \mu_j| < \min\left\{1, \frac{\varepsilon}{8|\beta^k_i|+8}\left(\frac{\tau-1}{\tau}\right), \frac{\varepsilon}{8|\mu_i|+8}\left(\frac{\tau-1}{\tau}\right)\right\}.
\end{equation}
Hence we can use Lemma \ref{lemma:q}. Thus $|Q_{ij} - \widetilde{Q}_{ij}| < \varepsilon$.

\end{proof}

Corollary \ref{cor:mu} shows that if one desires to know the mean of the model error to an accuracy of $\varepsilon$, then the analysis of that variable must be within $\frac{\varepsilon}{L+1}$ of the truth at every timestep. Contrast this with the result from Theorem \ref{thm:q}: to estimate the variance of a variable to within $\varepsilon$, the analysis of that variable must be within
\begin{equation}
\frac{1}{L+1}\min \left\{ 1, \frac{\varepsilon}{8|\beta^k_i|+8}\left(\frac{\tau-1}{\tau}\right), \frac{\varepsilon}{8|\mu_i|+8}\left(\frac{\tau-1}{\tau}\right)\right\}
\end{equation}
of the truth at each timestep. With the factor of $8$ in the denominator, this can be seen as approximately an order of magnitude more accuracy needed in the analysis to achieve the same accuracy in the approximation. Moreover, the accuracy in the variance depends on the mean value of that variable. The larger the mean value, the more accuracy required in the analysis to give the same error in the variance. This may be more easily understood conversely: the larger the mean value in a variable, the larger the error in the estimation of its variance.

\section{Numerical experiments}
In this section we apply the model error moment estimation method as detailed in the previous sections to the Lorenz 96 system \citep{Lorenz1996}. This is a very simple example to show numerically that the errors in the approximations of $\mu$ and $Q$ are dependent on the quality of the analysis $x_a$. The Lorenz 96 system is given by the following ODE:

\begin{equation}
\frac{\mathrm{d}x_i}{\mathrm{d}t} = -x_{i-2}x_{i-1} + x_{i-1}x_{i+1}-x_i + 8,
\end{equation}
where $i \in {1,\ldots,N}$ with cyclic boundary conditions. We choose $N = 40$.
We prescribe the truth to evolve as \eqref{eqn:truth} with $\beta \sim \mathcal{N}(\overline{\mu},\overline{Q})$ where
\begin{equation}
\overline{\mu}(x_i) = \frac{1}{5}\sin(\frac{\pi i}{N}),
\end{equation}
and $\overline{Q} = \left(\frac{1}{10}\right)^2\frac{3}{2}C$ where
\begin{equation}
C_{ij} = \begin{cases}
1 & \text{if } i=j\\
\frac{2}{3} & \text{if } |i-j|=1 \\
\frac{1}{6} & \text{if } |i-j|=2 \\
0 & \text{otherwise}
\end{cases}.
\end{equation}
$\overline{\mu}$ and $\overline{Q}$ have been chosen to be \emph{interesting} and $\overline{Q}$ in particular, for implementation purposes, has a rather simple symmetric square root.

We take observations of the true state at every timestep. The observations are related to the state via the equation
\begin{equation}
y^k = H(x^k_t) + \eta^k \qquad \text{where} \qquad \eta \sim \mathcal{N}(0,R).
\end{equation}
Here we consider $H=I$ and will show two different experiments with observations of varying accuracy. The first experiment will have $R=10^{-8}I$, and represents very informative observations. The second experiment will have $R=10^{-3}I$, and represents observations containing less information.

To get an analysis $x_a$, we perform a simple 3DVar at each timestep. That is, 
\begin{equation}
x_a^k = \arg\min_x \frac{1}{2}(x-f(x_a^{k-1}))^TB^{-1}(x-f(x_a^{k-1})) + \frac{1}{2}(y^k-H(x))^TR^{-1}(y^k-H(x)).
\end{equation}
We use $B=10^{12}I$ as a particularly uninformative prior. The model is evolved for 3000 timesteps, giving us 2999 samples of $\widetilde{\beta}^k$. Figure \ref{fig:hov} shows Hovm\"oller diagrams of the analysis error as a result of the data assimilation process. The higher precision observations lead to significantly lower estimation errors.

Using these analyses, estimates for $\widetilde{\mu}$ and $\widetilde{Q}$ are calculated. Results for $R=10^{-8}$ are shown in Figures \ref{fig:mu_8} to \ref{fig:Qe_8} and for $R=10^{-3}$ are shown in Figures \ref{fig:mu_3} to \ref{fig:Qe_3}. These results show that the second order moment of model error, $Q$, is much more sensitive to the quality of the analysis than the first order moment, $\mu$. This is consistent with the theory set out in the previous section.

\begin{figure}[h]
\begin{subfigure}[t]{0.48\textwidth}
\centering\caption{$R=10^{-8}$ $\downarrow$}
\includegraphics[width=\textwidth]{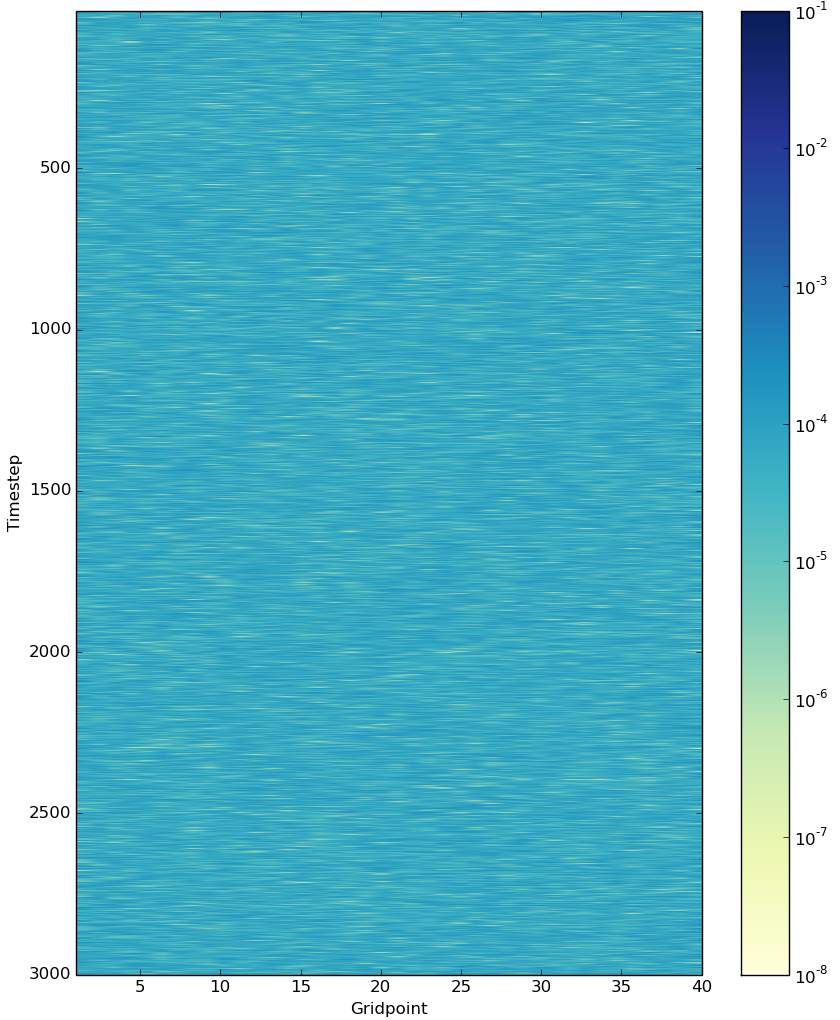}
\end{subfigure}
\begin{subfigure}[t]{0.48\textwidth}
\centering\caption{$R=10^{-3}$ $\downarrow$}
\includegraphics[width=\textwidth]{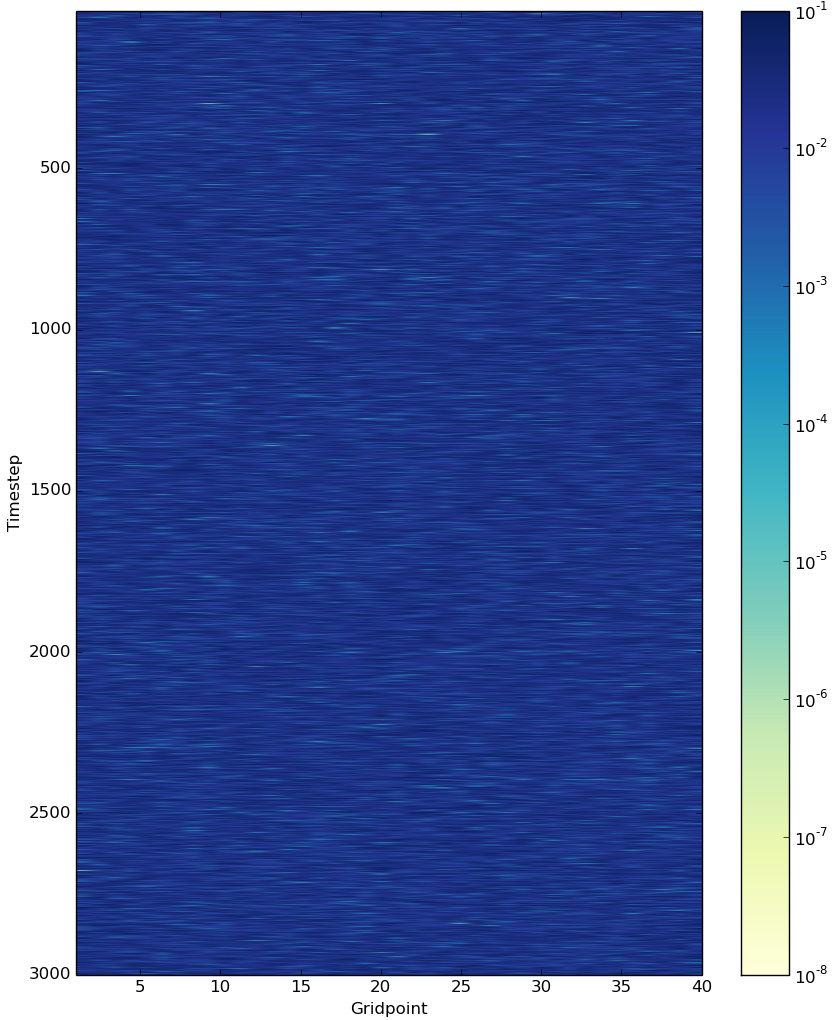}
\end{subfigure}
\caption{Hovm\"oller plots of $|x_t - x_a|$ for the two different precision observations. Observe that the 3DVar algorithm can reduce the error in the state by $2-3$ orders of magnitude when the observations are more precise.}\label{fig:hov}
\end{figure}

\begin{figure}[h]
\begin{subfigure}[t]{0.48\textwidth}
\centering\includegraphics{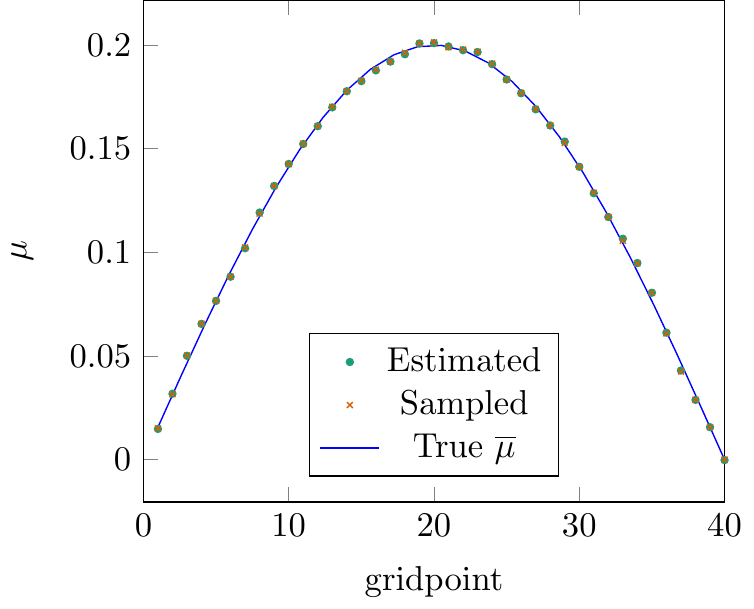}
\caption{Estimated, sampled, and true mean of model error}\label{fig:mu}
\end{subfigure}
\begin{subfigure}[t]{0.48\textwidth}
\centering\includegraphics{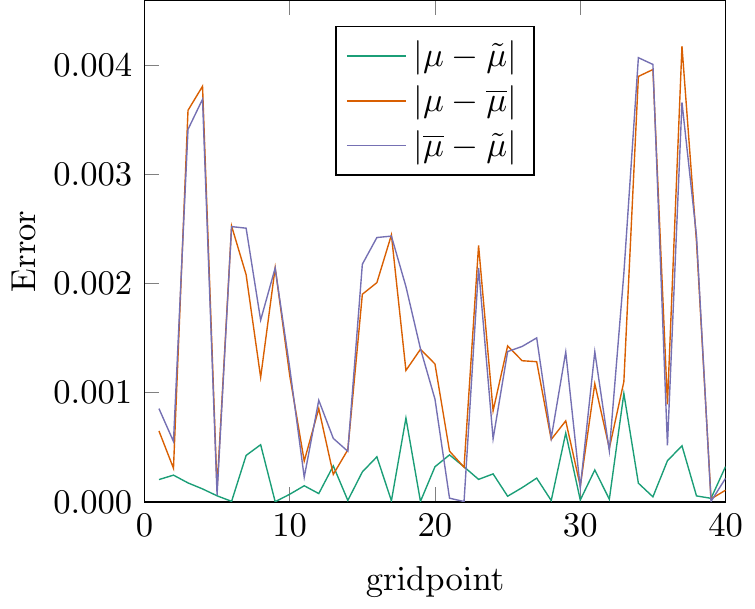}
\caption{Errors in the mean}
\end{subfigure}

\caption{The first moment of model error for $R=10^{-8}$}\label{fig:mu_8}
\end{figure}

\begin{figure}
\begin{subfigure}{0.33\textwidth}
\caption{True $\overline{Q}$ matrix}
\includegraphics[width=1.13\textwidth]{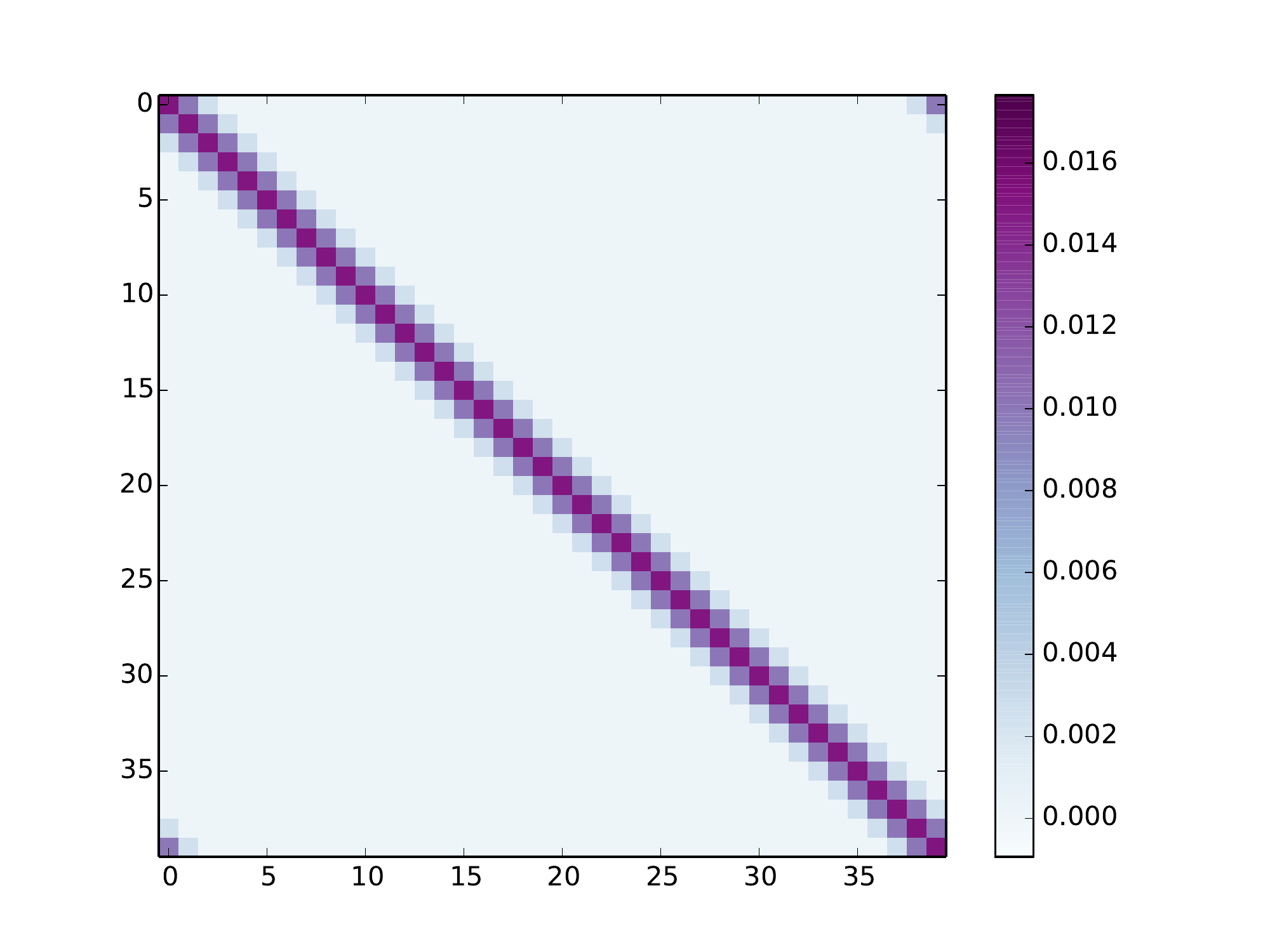}
\end{subfigure}
\begin{subfigure}{0.33\textwidth}
\caption{Sampled $Q$ matrix}
\includegraphics[width=1.13\textwidth]{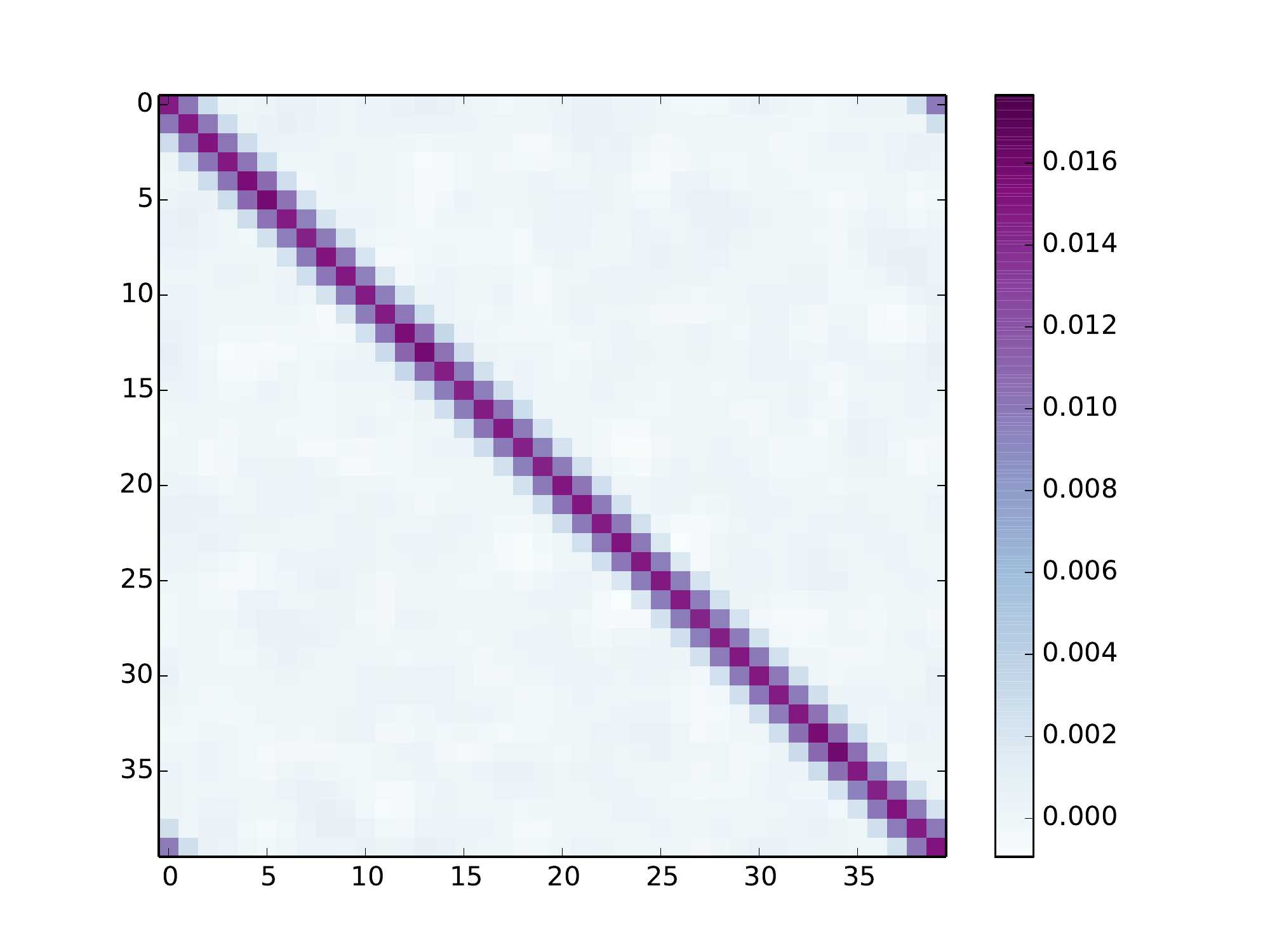}
\end{subfigure}
\begin{subfigure}{0.33\textwidth}
\caption{Estimated $\widetilde{Q}$ matrix}
\includegraphics[width=1.13\textwidth]{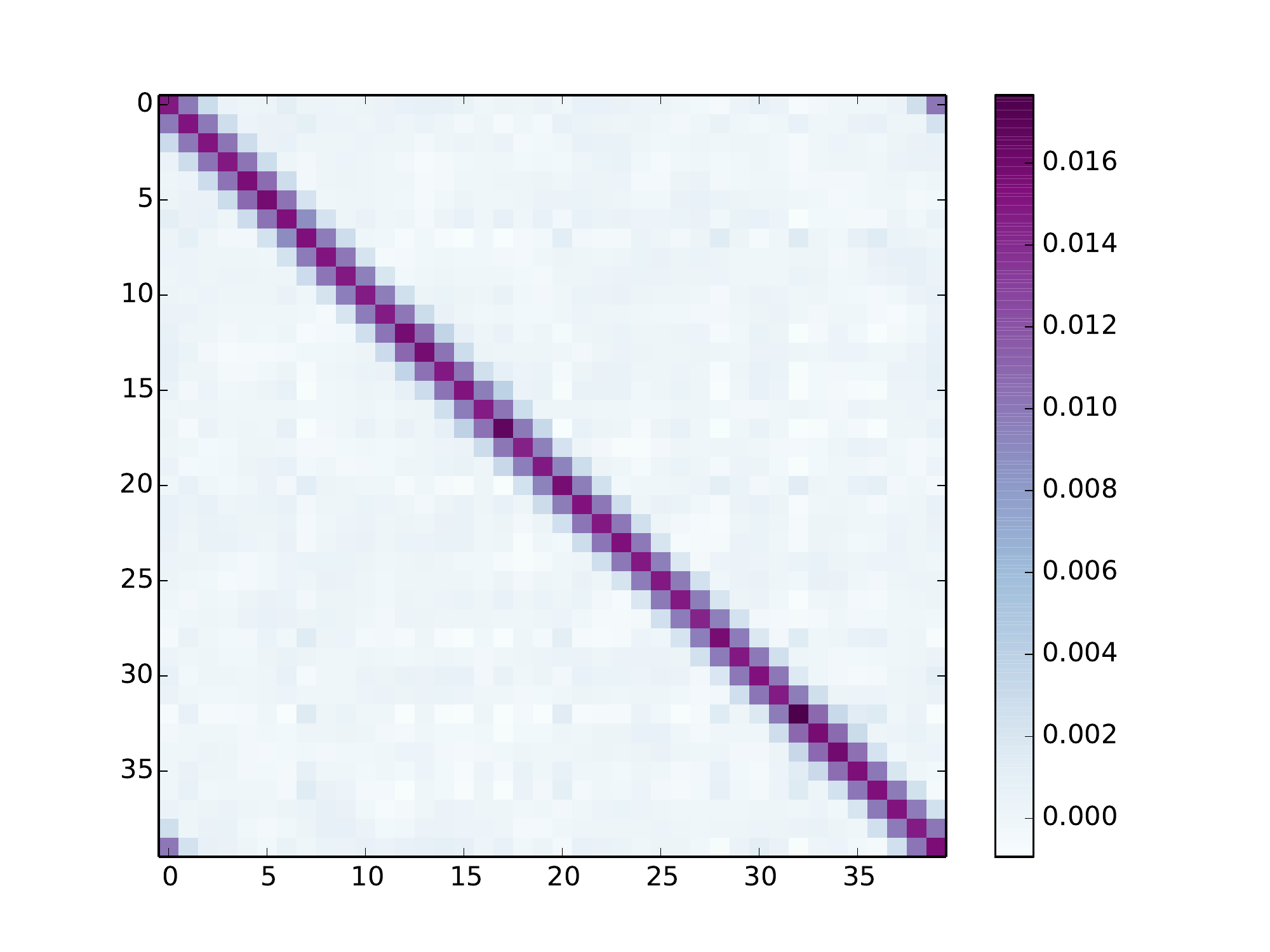}
\end{subfigure}
\caption{True, sampled, and estimated covariance of the model error $R=10^{-8}$}\label{fig:Q_8}
\end{figure}

\begin{figure}\captionsetup{justification=centering}
\begin{subfigure}[b]{0.33\textwidth}
\centering\caption{$|Q-\overline{Q}|$ \\showing the\\ sampling error}
\includegraphics[width=1.13\textwidth]{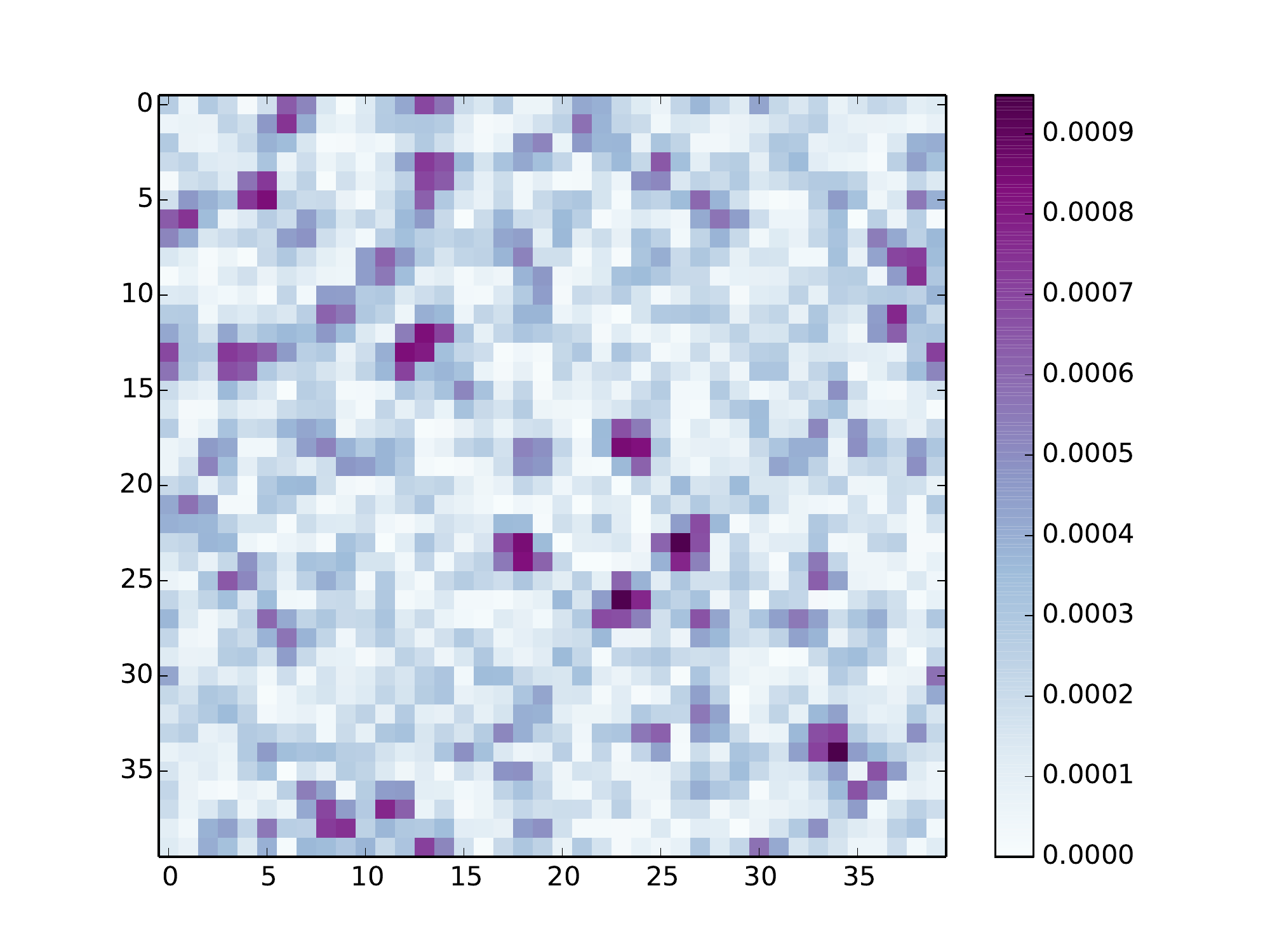}
\end{subfigure}
\begin{subfigure}[b]{0.33\textwidth}
\centering\caption{$|Q-\widetilde{Q}|$ \\showing the error in the estimate\\ and the samples which occurred}
\includegraphics[width=1.13\textwidth]{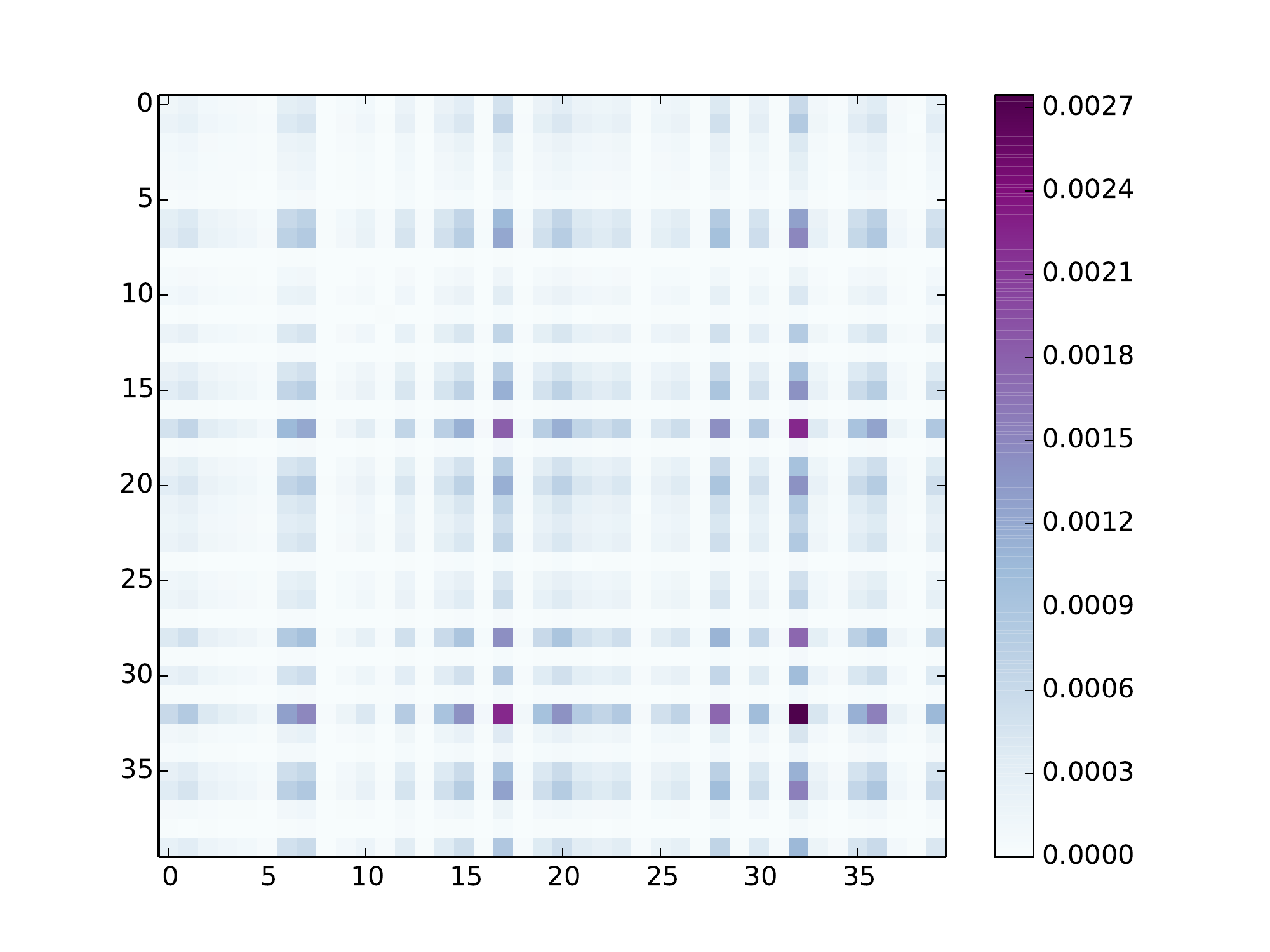}
\end{subfigure}
\begin{subfigure}[b]{0.33\textwidth}
\centering\caption{$|\overline{Q}-\widetilde{Q}|$\\ showing the error in the estimation\\ and the underlying true covariance}
\includegraphics[width=1.13\textwidth]{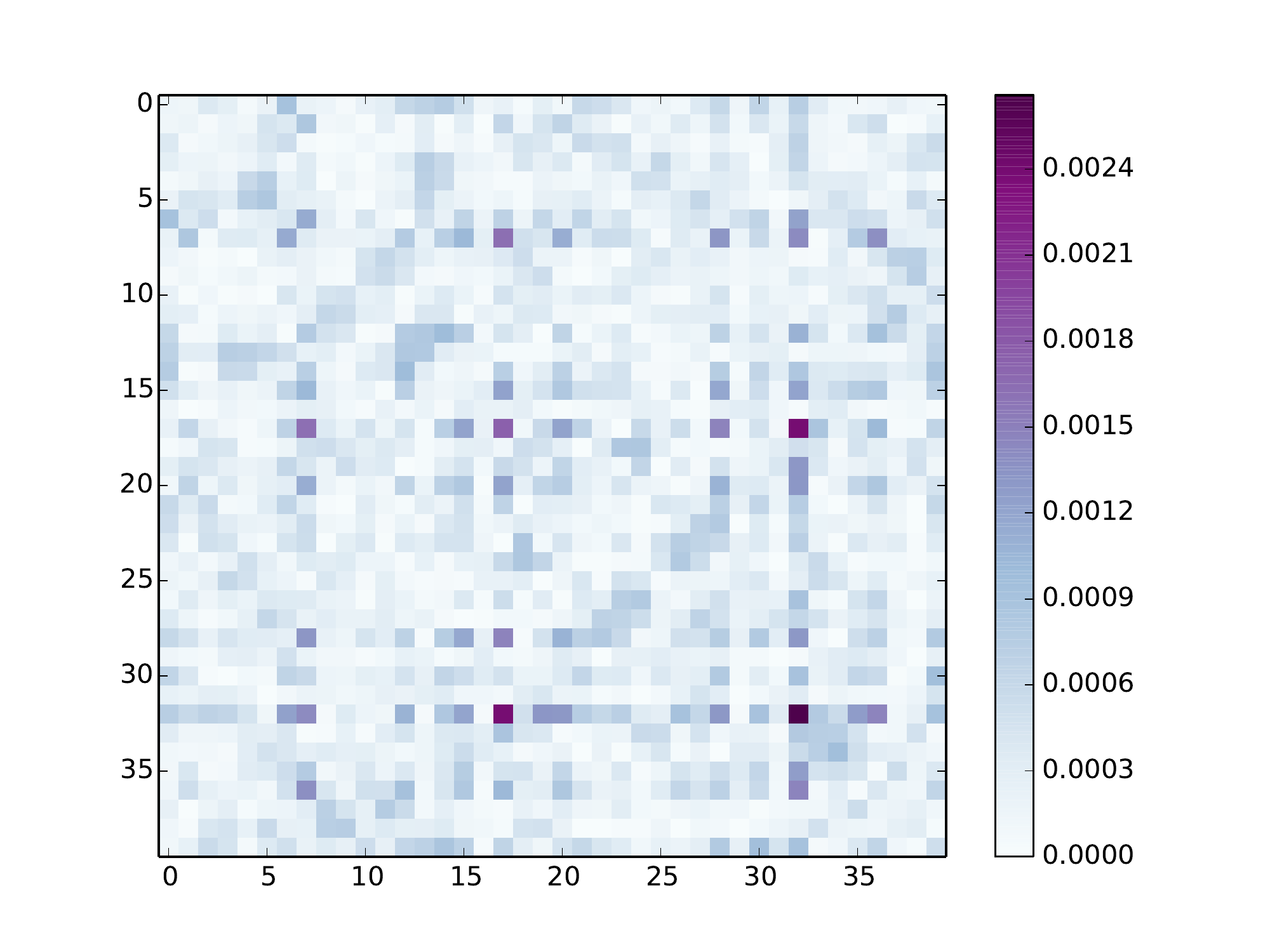}
\end{subfigure}
\caption{Errors in the model error covariance matrix $R=10^{-8}$}\label{fig:Qe_8}
\end{figure}

\begin{figure}[h]
\begin{subfigure}[t]{0.48\textwidth}
\centering\includegraphics{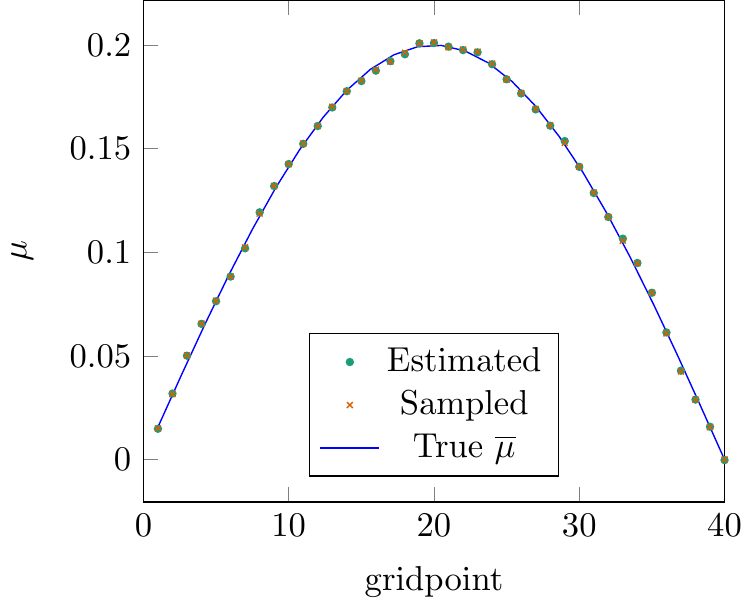}
\caption{Estimated, sampled, and true mean of model error}\label{fig:mu}
\end{subfigure}
\begin{subfigure}[t]{0.48\textwidth}
\centering\includegraphics{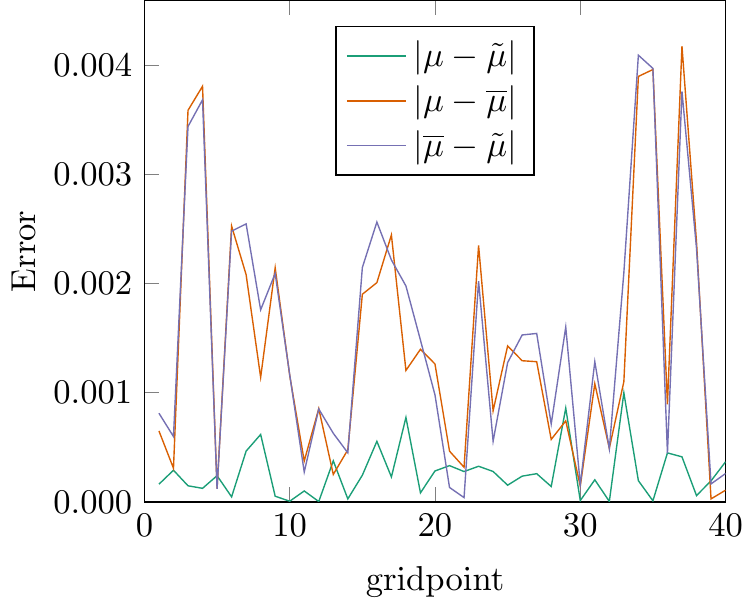}
\caption{Errors in the mean}
\end{subfigure}

\caption{The first moment of model error for $R=10^{-3}$}\label{fig:mu_3}
\end{figure}

\begin{figure}
\begin{subfigure}{0.33\textwidth}
\caption{True $\overline{Q}$ matrix}
\includegraphics[width=1.13\textwidth]{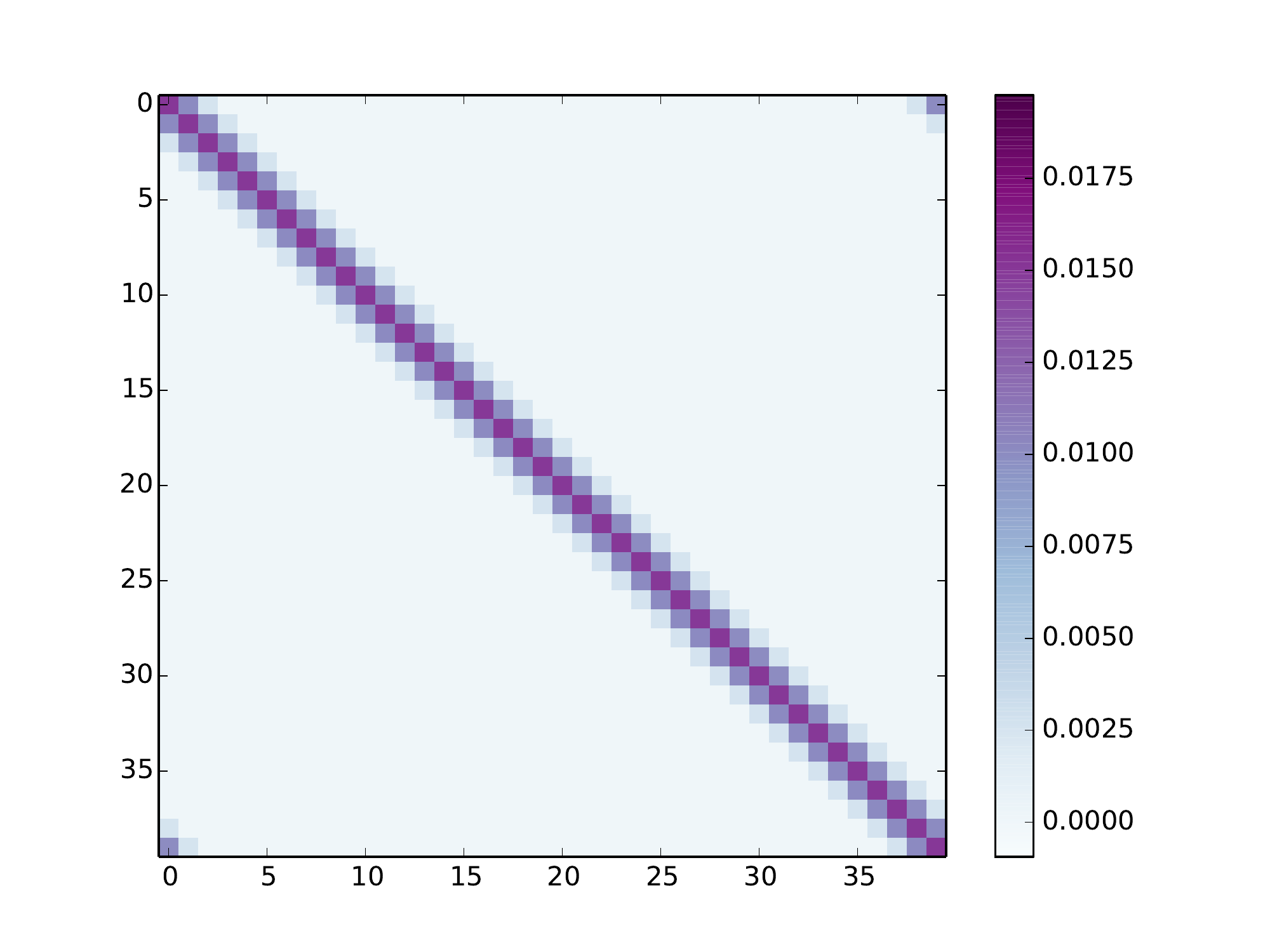}
\end{subfigure}
\begin{subfigure}{0.33\textwidth}
\caption{Sampled $Q$ matrix}
\includegraphics[width=1.13\textwidth]{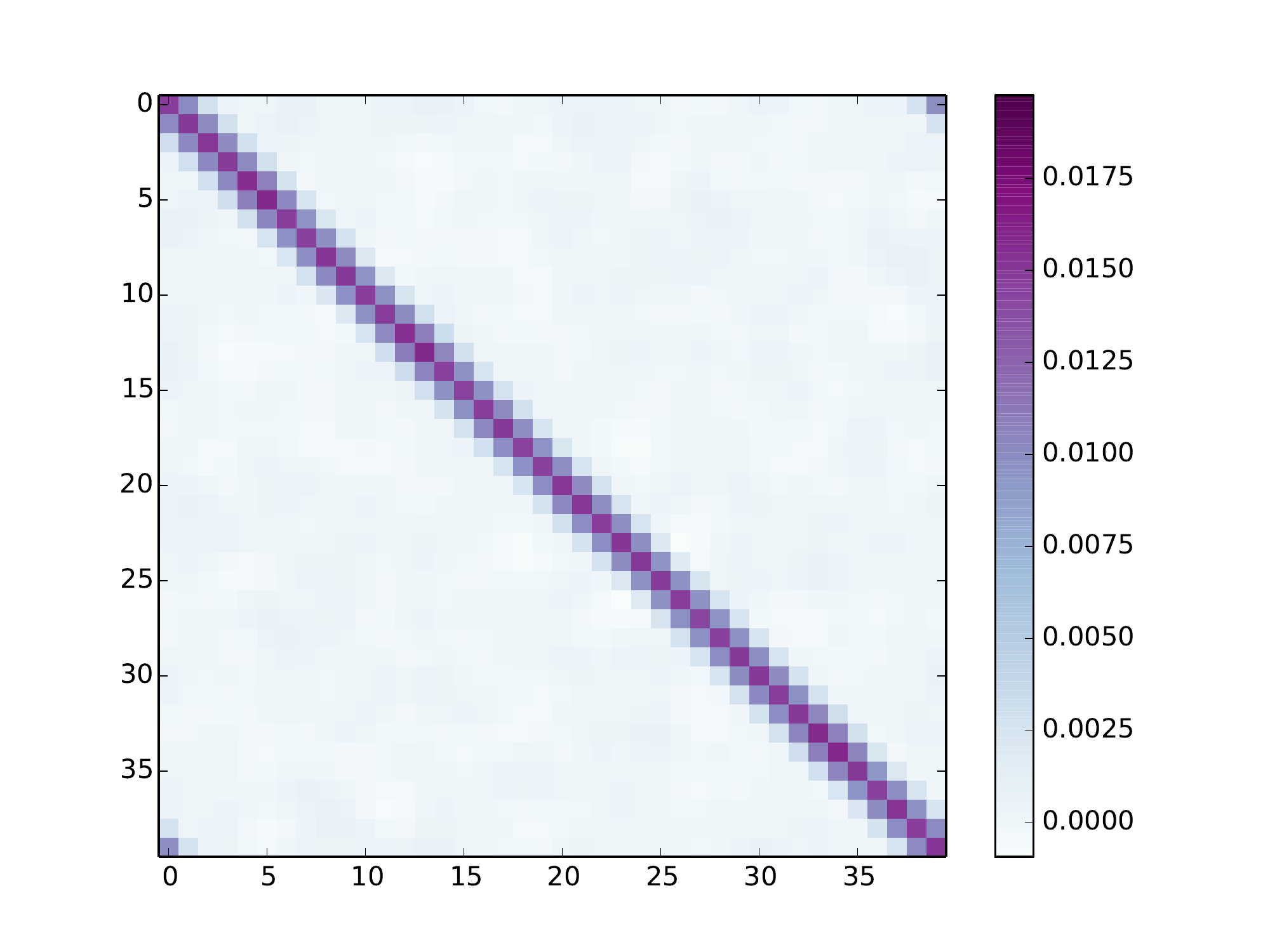}
\end{subfigure}
\begin{subfigure}{0.33\textwidth}
\caption{Estimated $\widetilde{Q}$ matrix}
\includegraphics[width=1.13\textwidth]{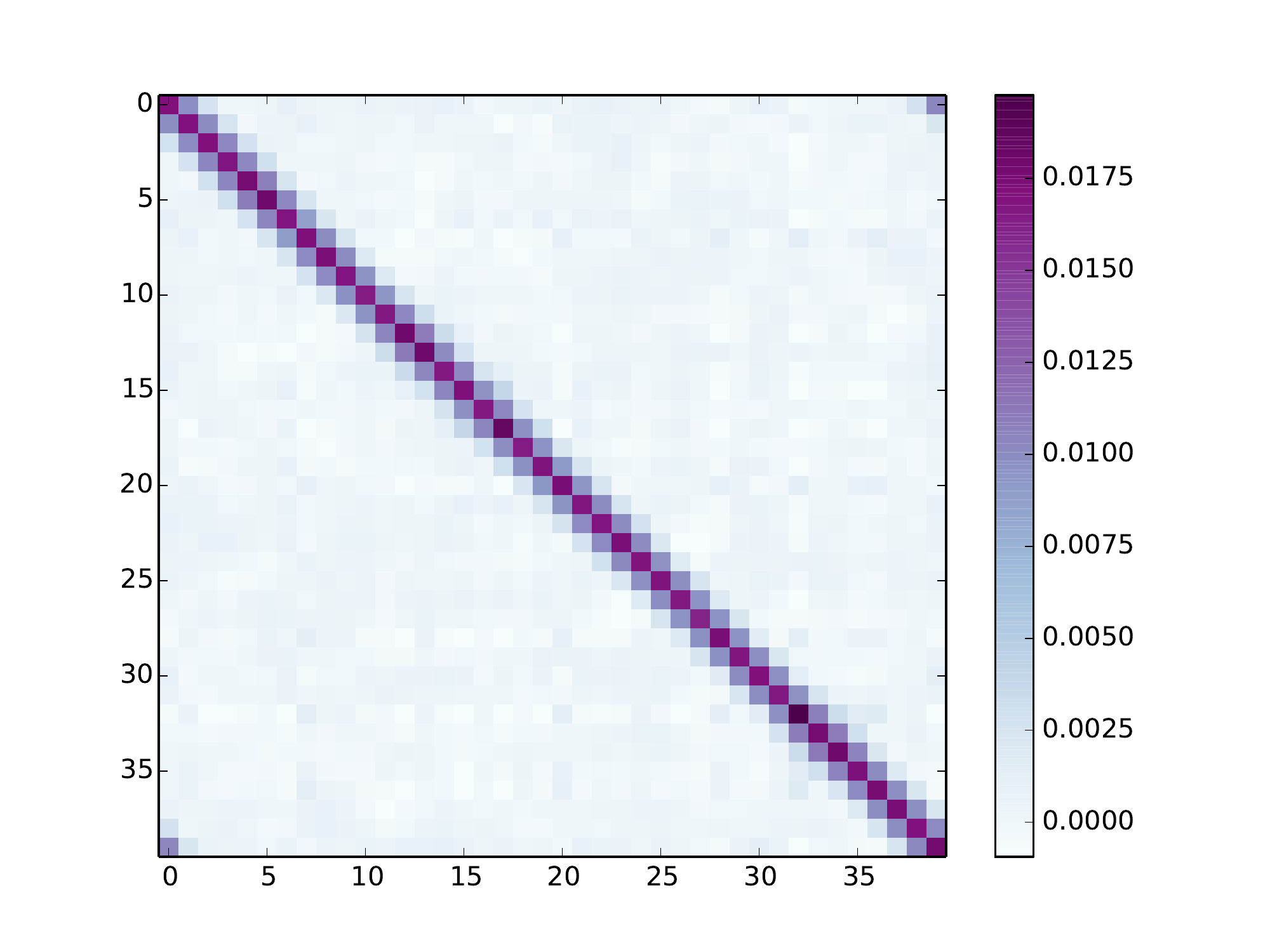}
\end{subfigure}
\caption{True, sampled, and estimated covariance of the model error $R=10^{-3}$}\label{fig:Q_3}
\end{figure}

\begin{figure}\captionsetup{justification=centering}
\begin{subfigure}[b]{0.33\textwidth}
\centering\caption{$|Q-\overline{Q}|$ \\showing the\\ sampling error}
\includegraphics[width=1.13\textwidth]{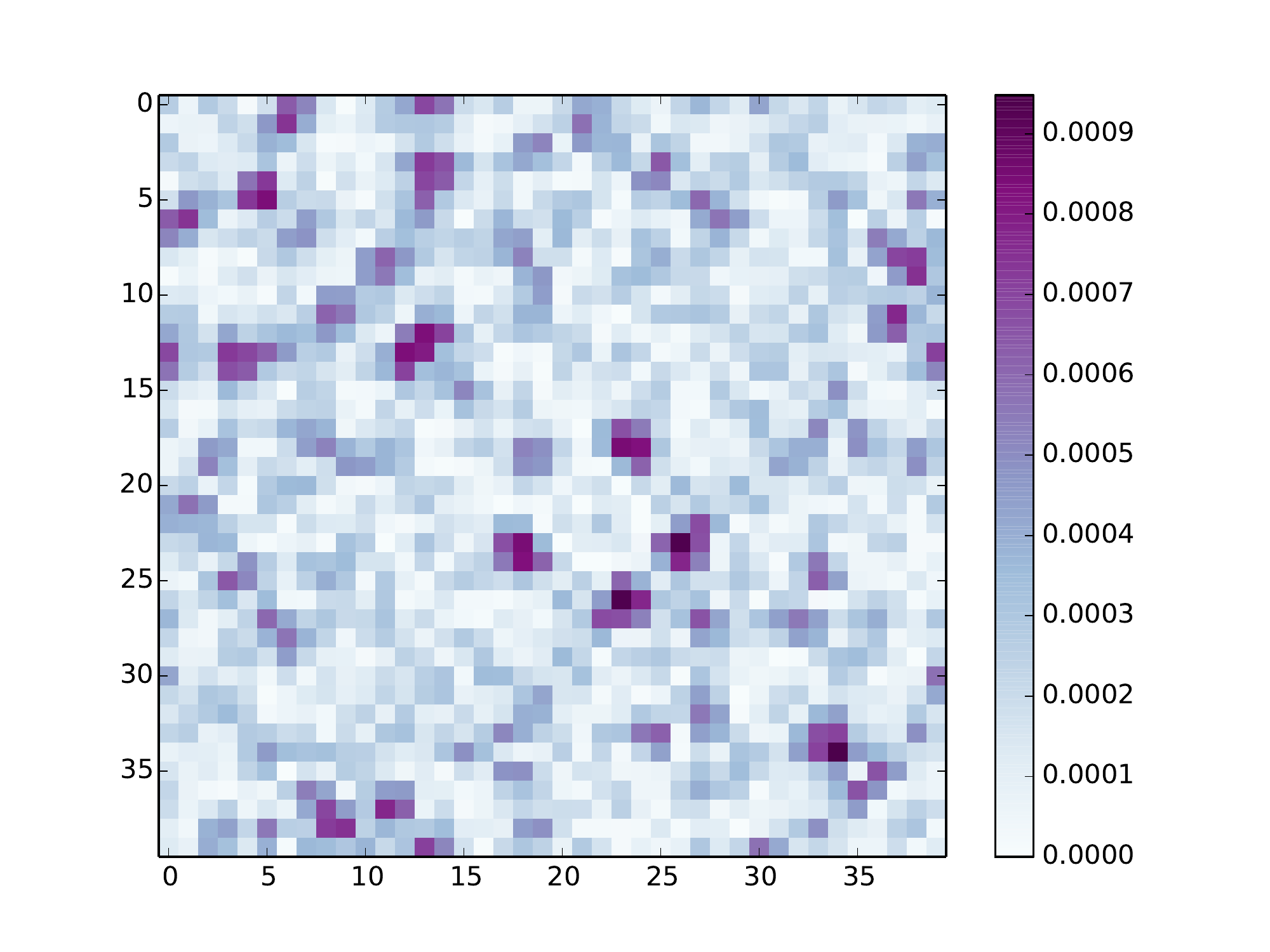}
\end{subfigure}
\begin{subfigure}[b]{0.33\textwidth}
\centering\caption{$|Q-\widetilde{Q}|$ \\showing the error in the estimate\\ and the samples which occurred}
\includegraphics[width=1.13\textwidth]{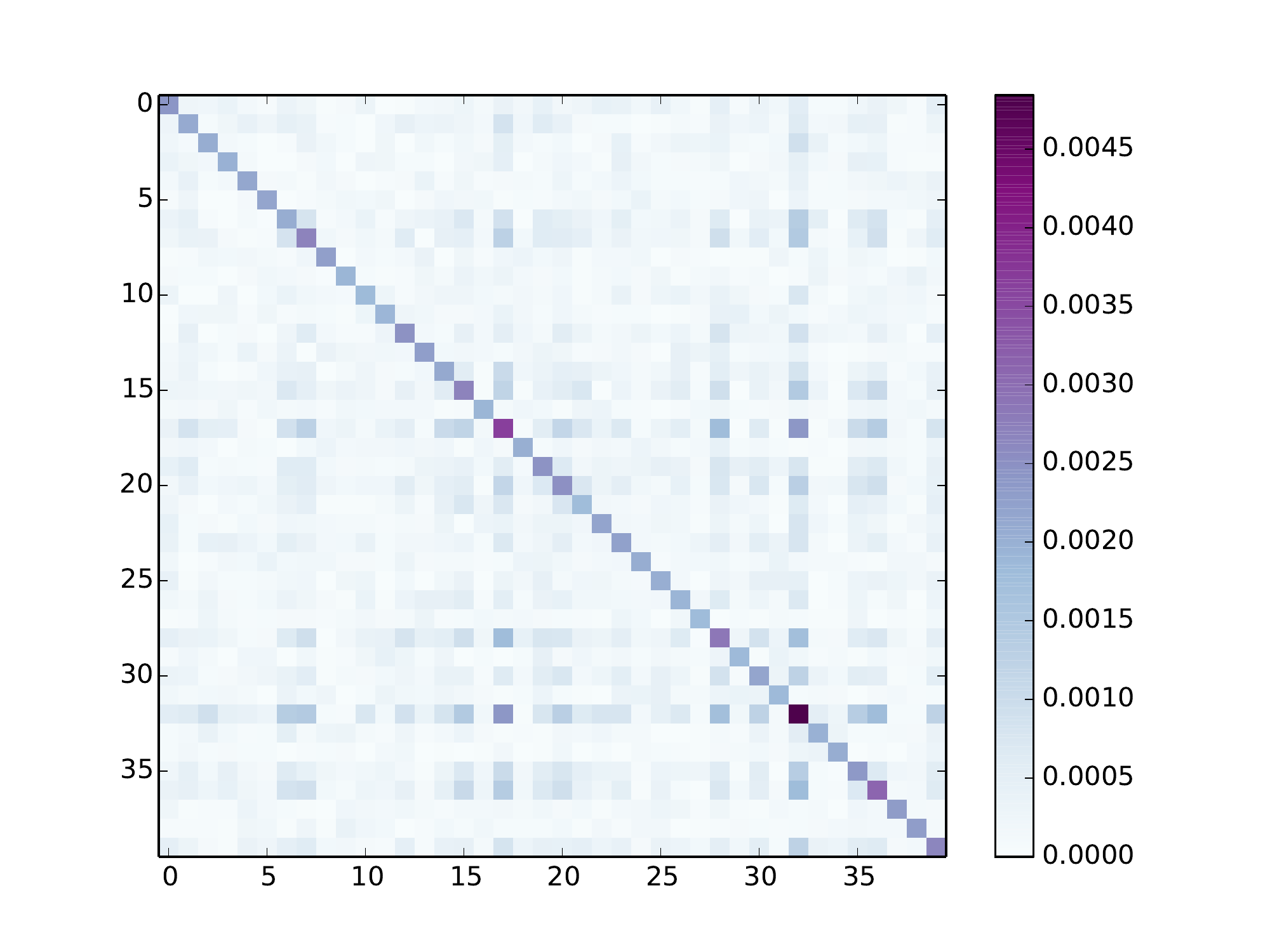}
\end{subfigure}
\begin{subfigure}[b]{0.33\textwidth}
\centering\caption{$|\overline{Q}-\widetilde{Q}|$\\ showing the error in the estimation\\ and the underlying true covariance}
\includegraphics[width=1.13\textwidth]{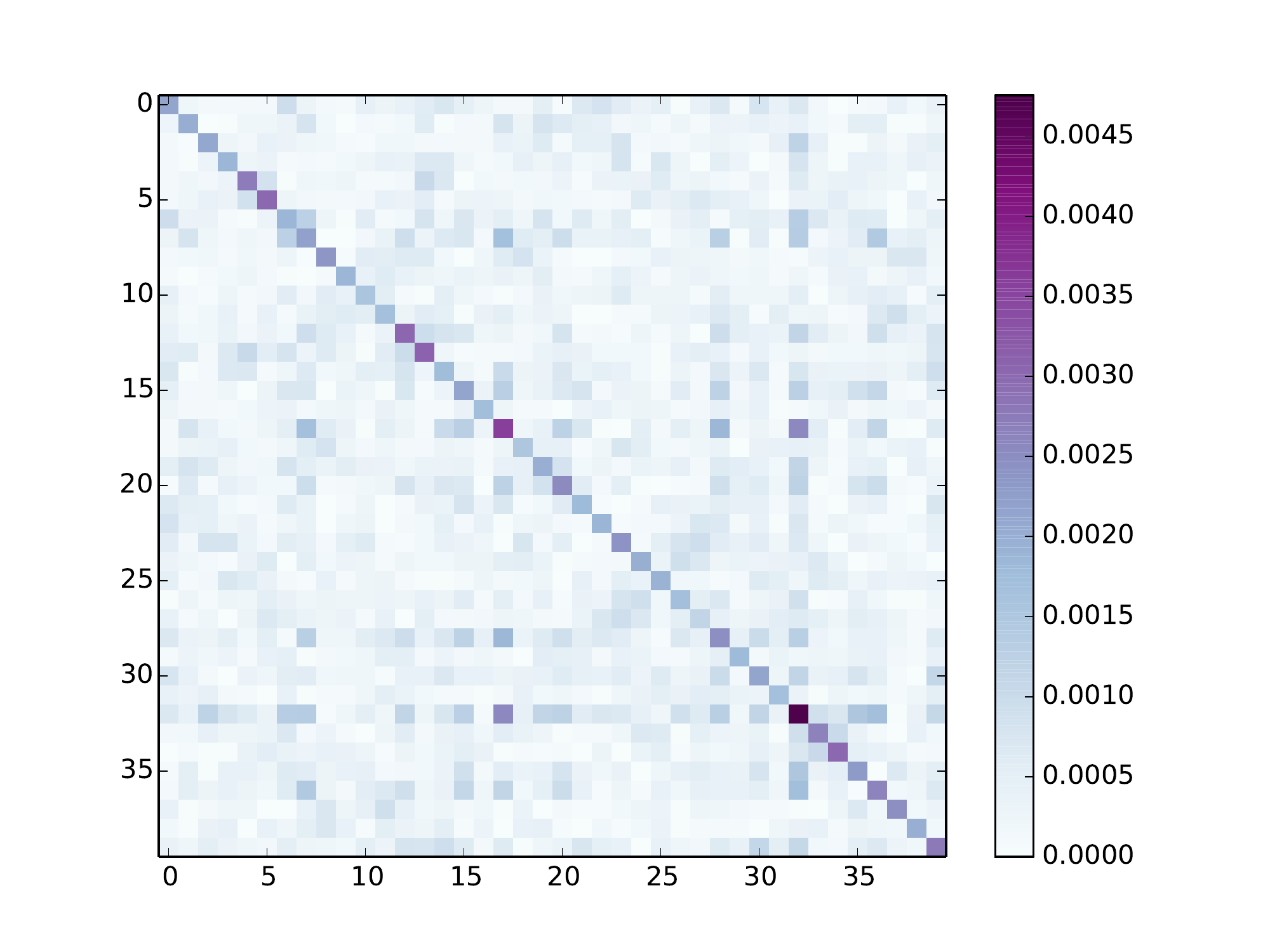}
\end{subfigure}
\caption{Errors in the model error covariance matrix $R=10^{-3}$}\label{fig:Qe_3}
\end{figure}

\section{Conclusions}
This paper has considered using the problem of estimating the moments of model error of a dynamic model. The model error is approximated by the difference of the analysis and the model forecast of the analysis at the previous timestep. Bounds for the errors in both the first and second moment of the approximated model error compared to the sample model error are derived, in terms of the error in the analysis from the truth. It is shown that to achieve the same error estimation in the second moment compared with the first, the analysis must be an order of magnitude closer to the truth.

Numerical experiments were conducted to elucidate how the quality of the analysis affects the estimation of both the mean and the covariance of the model error. It is shown numerically with the Lorenz 96 system that the estimation of the covariance of the model error is much more sensitive to the quality of the data assimilation than the estimation of the mean of the model error.

\clearpage
\bibliographystyle{apalike}

\begin{thebibliography}{}

\bibitem[Box and Draper, 1987]{Box1987}
Box, G.~E. and Draper, N.~R. (1987).
\newblock {\em Empirical model-building and response surfaces.}
\newblock John Wiley \& Sons.

\bibitem[Browne and van Leeuwen, 2015]{Browne2015a}
Browne, P. and van Leeuwen, P. (2015).
\newblock {Twin experiments with the equivalent weights particle filter and
  HadCM3}.
\newblock {\em Quarterly Journal of the Royal Meteorological Society}, 141(693
  October 2015 Part B):3399--3414.

\bibitem[Dee and Uppala, 2009]{Dee2009}
Dee, D.~P. and Uppala, S. (2009).
\newblock {Variational bias correction of satellite radiance data in the
  ERA-Interim reanalysis}.
\newblock {\em Quarterly Journal of the Royal Meteorological Society},
  135(October):1830--1841.

\bibitem[Evensen, 2007]{Evensen2007}
Evensen, G. (2007).
\newblock {\em Data assimilation}.
\newblock Springer.

\bibitem[Jazwinski, 1970]{Jazwinski1970}
Jazwinski, A.~H. (1970).
\newblock {\em Stochastic Processes and Filtering Theory}.
\newblock Academic Press.

\bibitem[Lang et~al., 2016]{Lang2016}
Lang, M., van Leeuwen, P.~J., and Browne, P. (2016).
\newblock {A systematic method of parameterisation estimation using data
  assimilation}.
\newblock {\em Tellus A}, 68:1--10.

\bibitem[Law et~al., 2015]{law2015}
Law, K., Stuart, A., and Zygalakis, K. (2015).
\newblock {\em Data Assimilation: A Mathematical Introduction}.
\newblock Texts in Applied Mathematics. Springer International Publishing.

\bibitem[Lorenz, 1996]{Lorenz1996}
Lorenz, E. (1996).
\newblock {Predictability: A problem partly solved}.
\newblock {\em Proc. Seminar on predictability}, 1(1):40--58.

\bibitem[Owhadi et~al., 2013]{Owhadi2013}
Owhadi, H., Scovel, C., Sullivan, T.~J., McKerns, M., and Ortiz, M. (2013).
\newblock {Optimal Uncertainty Quantification}.
\newblock {\em SIAM Review}, 55(2):271--345.

\bibitem[Teckentrup et~al., 2013]{Teckentrup2013}
Teckentrup, A.~L., Scheichl, R., Giles, M.~B., and Ullmann, E. (2013).
\newblock {Further analysis of multilevel Monte Carlo methods for elliptic PDEs
  with random coefficients}.
\newblock {\em Numerische Mathematik}, pages 24--29.

\bibitem[Todling, 2015]{Todling2014}
Todling, R. (2015).
\newblock {A lag-1 smoother approach to system-error estimation: sequential
  method}.
\newblock {\em Quarterly Journal of the Royal Meteorological Society},
  141(690):1502--1513.

\bibitem[Tr\'emolet, 2006]{Tremolet2006}
Tr\'emolet, Y. (2006).
\newblock {Accounting for an imperfect model in 4D-Var}.
\newblock {\em Quarterly Journal of the Royal Meteorological Society},
  132(621):2483--2504.

\bibitem[Tr\'emolet, 2007]{Tremolet2007}
Tr\'emolet, Y. (2007).
\newblock {Model-error estimation in 4D-Var}.
\newblock {\em Quarterly Journal of the Royal Meteorological Society},
  1280(July):1267--1280.

\bibitem[van Leeuwen, 2010]{VanLeeuwen2010}
van Leeuwen, P. (2010).
\newblock {Nonlinear data assimilation in geosciences: an extremely efficient
  particle filter}.
\newblock {\em Quarterly Journal of the Royal Meteorological Society},
  136(653):1991--1999.

\end{thebibliography}

\end{document}